\documentclass[reqno]{amsart}
\usepackage{amssymb,amsmath,amsthm,newlfont,enumerate}

\theoremstyle{plain}
\newtheorem{theorem}{Theorem}[section]
\newtheorem{lemma}[theorem]{Lemma}
\newtheorem{corollary}[theorem]{Corollary}

\theoremstyle{definition}

\theoremstyle{remark}

\newcommand{\CC}{{\mathbb C}}

\newcommand{\cP}{{\mathcal P}}

\begin{document}

\title[Lower bound for capacity]{A least-squares
lower bound for capacity}

\author[T. Ransford]{Thomas Ransford}

\address{%
D\'epartement de math\'ematiques et de statistique, Universit\'e Laval, Qu\'ebec (QC), G1V 0A6, Canada}

\email{ransford@mat.ulaval.ca}

\thanks{Research supported by  NSERC Discovery Grant RGPIN-2026-04565}



\date{17 October 2026}

\begin{abstract}
We establish a least-squares lower bound for the logarithmic
capacity of an arbitrary compact plane set, 
extending an inequality previously known for connected sets.
The bound  generalizes classical inequalities 
involving area and moment of inertia. It is sharp for disks, 
and asymptotically sharp for compact sets whose capacity
coincides with that of their interior. We further show that
asymptotic sharpness holds for general compact sets if
we replace logarithmic capacity by minimal carrier capacity.
\end{abstract}

\maketitle

\section{Introduction}

Logarithmic capacity measures the size of a planar set from the
point of view of two-dimensional potential theory. It also plays a key role in complex analysis and approximation theory.
For background information 
see, for example, \cite[Chapter~5]{Ra95}.

There are several equivalent ways of defining  capacity.
For the purposes of this article, it is convenient to adopt
a characterization in terms of norms of polynomials.
Fix the following notation.
If $K$ is a non-empty compact subset of $\CC$
and $f:K\to\CC$ is a continuous function, then we
write
\[
\|f\|_{L^\infty(K)}:=\max_{z\in K}|f(z)|.
\]
Let $\cP_n$ be the set of monic complex polynomials
of degree~$n$.
For $n\ge1$, define
\[
m_n(K)=\inf_{p\in\cP_n}\|p\|_{L^\infty(K)}.
\]
Then the logarithmic capacity of $K$ can be characterized
by the relations
\begin{equation}\label{E:Chebyshev}
c(K)=\lim_{n\to\infty}m_n(K)^{1/n}=\inf_{n\ge1}m_n(K)^{1/n}.
\end{equation}
See, for example, \cite[Corollary~5.5.5]{Ra95}.

Our goal in this article is to establish least-squares lower bounds
for $c(K)$ using the norm
\[
\|f\|_{L^2(K)}:=\Bigl(\int_K|f(z)|^2\,dA(z)\Bigr)^{1/2},
\]
where  $dA=dx\,dy$ is planar area measure.
This is potentially of interest,
among other reasons, because
least-squares minima are easy to calculate.
An analysis of the sharpness of these inequalities will
eventually lead us to a dual version of \eqref{E:Chebyshev}.

The first result is our basic least-squares lower bound for capacity.

\begin{theorem}\label{T:bound}
For every compact plane set $K$ and every integer $n\ge1$,
\begin{equation}\label{E:bound}
c(K)^n\ge \frac{1}{\sqrt{\pi n}}\inf_{p\in\cP_n}\|p'\|_{L^2(K)}.
\end{equation}
\end{theorem}

When $n=1$, the inequality \eqref{E:bound} reduces to a 
bound for capacity in terms of the area $A(K)$ of $K$, namely
\begin{equation}\label{E:A(K)}
c(K)\ge \sqrt{\frac{A(K)}{\pi}}.
\end{equation}
This result is  well known, see for example
\cite[Theorem~5.2.5]{Ra95}.

If $n=2$, then \eqref{E:bound} is equivalent to the following bound for $c(K)$ in 
terms of the moment of inertia  $I(K):=\inf_{a\in\CC}\int_K|z-a|^2\,dA(z)$:
\begin{equation}\label{E:I(K)}
c(K)\ge \Bigl(\frac{2I(K)}{\pi}\Bigr)^{1/4}.
\end{equation}
The inequality \eqref{E:I(K)} is stronger than \eqref{E:A(K)}, and is of
interest because it takes into account the dispersion of $K$.
For connected  $K$, it was proved by P\'olya and Szeg\H o in
\cite[\S5.9]{PS51}, and the general case was established
by Hansen and Nadirashvili \cite[Theorem~3]{HN92}
using a clever symmetrization argument.
For more on the relationship between capacity and moment of inertia,
see Laugesen's article \cite{La22}.

In \cite{Ra94}, the author adapted the technique of 
P\'olya--Szeg\H o to establish \eqref{E:bound} for all $n\ge1$ in the case when
$K$ is connected, and posed the problem of proving
it in the general case. Theorem~\ref{T:bound} solves this problem.

If $K$ is a closed disk of radius $r$,  then $c(K)=r$
and $\inf_{p\in \cP_n}\|p'\|^2_{L^2(K)}=\pi n r^{2n}$,
so equality holds in \eqref{E:bound}. In particular,
this shows that the constant $1/\sqrt{\pi n}$ is sharp.

However, for many sets $K$,
the lower bound in \eqref{E:bound} is far from optimal.
For example, if $K$ is a line segment of length $\ell$,
then $c(K)=\ell/4$, whereas the right-hand side of
\eqref{E:bound} is clearly equal to zero for all~$n\ge1$
since $K$ has area zero.
This example suggests that, to obtain a sharper result, 
we need to replace the logarithmic capacity $c(K)$
by a version of capacity that does not notice sets of area zero.

Such a notion already exists, namely, the 
\emph{minimal carrier capacity} of the measure $1_K\,dA$
(see \cite[Definition~1.1.1]{ST92}).
It can be defined as 
\begin{equation}\label{E:cAdef}
c_A(K):=\inf\bigl\{c(E): \text{Borel~}E\subset K,\, A(K\setminus E)=0\bigr\},
\end{equation}
where, for a bounded Borel subset $E$ of $\CC$, we write $A(E)$
for the area of $E$ and $c(E)$ for the inner capacity of $E$,
namely $c(E):=\sup\{c(F): F\subset E,\, F\text{~compact}\}$.
With this definition, it is plausible that
\eqref{E:bound} can be strengthened to
\[
c_A(K)^n\ge \frac{1}{\sqrt{\pi n}}\inf_{p\in\cP_n}\|p'\|_{L^2(K)}.
\]
Our second result shows that this is indeed true, and
that the improved inequality is asymptotically sharp.

\begin{theorem}\label{T:asymp}
Let $K$ be a compact subset of $\CC$. For $n\ge1$, define
\[
b_n(K):=\frac{1}{\sqrt{n\pi}}\inf_{p\in\cP_n}\|p'\|_{L^2(K)}.
\]
Then
\begin{equation}\label{E:asymp}
\lim_{n\to\infty}b_n(K)^{1/n}=\sup_{n\ge1}b_n(K)^{1/n}=c_A(K).
\end{equation}
\end{theorem}

We shall see that $c(K^\circ)\le c_A(K)\le c(K)$,
where $K^\circ$ denotes the interior of $K$.
Thus we obtain the following corollary,
which can be viewed as a dual
form of the classical relation \eqref{E:Chebyshev}.

\begin{corollary}\label{C:asymp}
If $c(K^\circ)=c(K)$, then
\begin{equation}\label{E:asymp2}
\lim_{n\to\infty}b_n(K)^{1/n}=\sup_{n\ge1}b_n(K)^{1/n}=c(K).
\end{equation}
\end{corollary}

The asymptotic behaviour of the quantities appearing in 
Theorem~\ref{T:asymp} and Corollary~\ref{C:asymp}
is closely related to the theory of orthogonal polynomials. Indeed, the squares of the relevant least-squares quantities are constant multiples of the squared \(L^2\)-norms of monic orthogonal polynomials with respect to planar area measure. 
General asymptotic estimates involving minimal carrier capacity were established by Stahl and Totik \cite{ST92}. For more detailed results concerning Bergman polynomials on unions of planar domains, see \cite{GPSS09} and \cite{SSST15}.

%%%%%%%%%%%%%%%%%%%%%%%%%%%%%%%%%%

\section{Proof of Theorem~\ref{T:bound}}

The proof proceeds via two lemmas.

\begin{lemma}\label{L:Greenineq}
Let $P$ be a monic polynomial of degree $d\ge1$,
let $r>0$ and let $\Omega:=\{z\in\CC:|P(z)|>r^d\}$.
Let $F$ be a holomorphic function on $\Omega$
such that, for some $a\in\CC$ and $n\ge1$,
\begin{equation}\label{E:Fatinfinity}
F(z)=az^{-n}+O(|z|^{-n-1})
\quad(|z|\to\infty).
\end{equation}
Then
\begin{equation}\label{E:Greenineq}
\int_\Omega |F'(z)|^2\,dA(z)\ge \frac{n\pi |a|^2}{r^{2n}}.
\end{equation}
\end{lemma}

\begin{proof}
We can suppose, without loss of generality, that
$F$ is holomorphic on a neighbourhood of $\overline{\Omega}$
and that $P'$ has no zero on the level set $|P|=r^d$.
Indeed, once the lemma is proved in this case, we can
apply it on $\{z:|P(z)|>r_m^d\}$ for a suitable sequence
$r_m$ decreasing to $r$, and then let $m\to\infty$.

Since $P'$ has no zero on $|P|=r^d$, the boundary of $\Omega$
consists of finitely many smooth closed curves.
If $T$ is chosen large enough so that the boundary of 
$\Omega$ is entirely contained in the disk $|z|<T$,
then $\Omega_T:=\Omega\cap\{|z|<T\}$ is admissible for Green's theorem. Thus, if $\phi,\psi$ are 
$C^\infty$ functions defined on a neighbourhood of
$\overline{\Omega}_T$, then
\begin{equation}\label{E:Green}
\int_{\Omega_T}(\psi\Delta\phi-\phi\Delta\psi)\,dA
=\int_{\partial\Omega_T}(\psi\partial_\nu\phi-\phi\partial_\nu\psi)\,ds,
\end{equation}
where $\partial_\nu$ denotes the directional derivative along the
outward unit normal from $\Omega_T$, and $ds$ is arclength
measure along $\partial\Omega_T$. We shall apply this formula
with various choices of $\phi$ and $\psi$.

Define $g$ on a neighbourhood of $\overline{\Omega}$ by
\[
g(z):=\frac{1}{d}\log|P(z)| -\log r.
\]
Then $\Delta g=0$  on a neighbourhood of $\overline{\Omega}$.
Also $g>0$ on $\Omega$ and $g=0$ on $\partial\Omega$. Further,
\begin{equation}\label{E:gatinfinity}
g(z)=\log|z|-\log r+O(|z|^{-1}) \quad(|z|\to\infty).
\end{equation}
Thus $g$ is the Green's function of $\Omega$ with pole at $\infty$.

Define
\[
\rho:=|F|^2e^{2ng}.
\]
By \eqref{E:Green} applied to suitable functions $\phi,\psi$, 
for all large enough $T$ we have
\begin{equation}\label{E:3equations}
\begin{aligned}
\int_{\Omega_T}(g\Delta\rho-\rho\Delta g)\,dA &=\int_{\partial\Omega_T}(g\partial_\nu\rho-\rho\partial_\nu g)\,ds,\\
\int_{\Omega_T}\Delta\rho\,dA&=\int_{\partial\Omega_T}\partial_\nu \rho\,ds,\\
\int_{\Omega_T} \Delta |F|^2\,dA&=\int_{\partial\Omega_T}\partial_\nu|F|^2\,ds.
\end{aligned}
\end{equation}

We now investigate what happens as $T\to\infty$. 
Using \eqref{E:Fatinfinity} and \eqref{E:gatinfinity},
it is elementary to check that, as $|z|\to\infty$,
\[
\rho(z)=|a|^2/r^{2n} +O(|z|^{-1}),
\quad
\nabla\rho(z)=O(|z|^{-2})
\quad\text{and}\quad
\Delta\rho(z)=O(|z|^{-4}).
\]
Hence, as $T\to\infty$, the left-hand sides in \eqref{E:3equations}
 simplify and converge as follows:
\begin{align*}
&\int_{\Omega_T}(g\Delta\rho-\rho\Delta g)\,dA
=\int_{\Omega_T}g\Delta\rho\,dA\to \int_{\Omega}g\Delta\rho\,dA, \\
&\int_{\Omega_T}\Delta\rho\,dA\to\int_{\Omega}\Delta\rho\,dA, \\
&\int_{\Omega_T}\Delta|F|^2\,dA=\int_{\Omega_T}4|F'|^2\,dA\to\int_{\Omega}4|F'|^2\,dA.
\end{align*}
As for the right-hand sides of \eqref{E:3equations}, the integrals over  $\partial\Omega$ can be developed as follows:
\begin{align*}
&\int_{\partial\Omega}(g\partial_\nu\rho-\rho\partial_\nu g)\,ds
=-\int_{\partial\Omega}\rho\partial_\nu g\,ds,\\
&\int_{\partial\Omega}\partial_\nu \rho\,ds=
\int_{\partial\Omega}\bigl(\partial_\nu|F|^2+2n|F|^2\partial_\nu g)\,ds
=\int_{\partial\Omega}\bigl(\partial_\nu|F|^2+2n\rho\partial_\nu g)\,ds.
\end{align*}
The integrals over the remaining part of $\partial\Omega_T$, namely the circle $|z|=T$,
behave as follows as $T\to\infty$:
\begin{align*}
&\int_{|z|=T}g\partial_\nu\rho\,ds=O(T^{-1}\log T),\\
&\int_{|z|=T}\rho\partial_\nu g\,ds=\int_{|z|=T}\frac{|a|^2}{r^{2n}}\frac{1}{T}\,ds+O(T^{-1})=\frac{2\pi |a|^2}{r^{2n}}+O(T^{-1}),\\
&\int_{|z|=T}\partial_\nu\rho\,ds=O(T^{-1}),\\
&\int_{|z|=T}\partial_\nu|F|^2\,ds=O(T^{-2n}).
\end{align*}
Therefore, letting $T\to\infty$ in \eqref{E:3equations}, we obtain the system
\begin{align*}
\int_{\Omega}g\Delta\rho\,dA &=-\int_{\partial\Omega}\rho\partial_\nu g\,ds-\frac{2\pi|a|^2}{r^{2n}}
\\
\int_{\Omega}\Delta\rho\,dA&=\int_{\partial\Omega}\bigl(\partial_\nu|F|^2+2n\rho\partial_\nu g)\,ds,\\
\int_{\Omega}4|F'|^2\,dA&=\int_{\partial\Omega}\partial_\nu|F|^2\,ds.
\end{align*}
Eliminating the boundary integrals, we deduce the following relation:
\[
\int_\Omega|F'|^2\,dA
=\frac{n\pi|a|^2}{r^{2n}}+\frac{1}{4}\int_\Omega (1+2ng)\Delta\rho\,dA.
\]
The integral on the right-hand side is non-negative, since $g\ge0$
on $\Omega$ and
\[
\Delta\rho=4e^{2ng}|F'+2ng_zF|^2\ge0,
\]
where $g_z:=(g_x-ig_y)/2$.
Therefore the inequality \eqref{E:Greenineq} holds.
\end{proof}

In the next result, we write $F_z:=(1/2)(F_x-iF_y)$
and $F_{\overline{z}}:=(1/2)(F_x+iF_y)$.

\begin{lemma}\label{L:energy}
Let $F:\CC\to\CC$ be a $C^\infty$-function such that
$F_{\overline{z}}$ has compact support
and $F(z)\to0$ as $|z|\to\infty$.
Then
\begin{equation}\label{E:energy}
\int_{\CC}|F_z|^2\,dA=\int_{\CC}|F_{\overline{z}}|^2\,dA.
\end{equation}
\end{lemma}

\begin{proof}
Both the integrals converge.
Indeed, 
$F_{\overline{z}}$ has compact support,
 and since $F$ is holomorphic outside a large disk and tends to zero at infinity, we have $F_z=O(|z|^{-2})$
at $\infty$.
By Stokes' theorem, for each $T>0$, we have
\begin{equation}\label{E:Stokes}
\int_{|z|<T}(|F_z|^2-|F_{\overline{z}}|^2)\,dA(z)
=\frac{1}{2i}\int_{|z|<T}d\overline{F}\wedge dF
=\frac{1}{2i}\int_{|z|=T} \overline{F}\,dF.
\end{equation}
As $|z|\to\infty$, we have $F(z)=O(|z|^{-1})$ and $F_z,F_{\overline{z}}=O(|z|^{-2})$, so
\[
\int_{|z|=T} \overline{F}\,dF=O(T^{-2}) \quad(T\to\infty).
\]
Letting $T\to\infty$ in \eqref{E:Stokes} therefore gives \eqref{E:energy}.
\end{proof}

\begin{proof}[Proof of Theorem~\ref{T:bound}]
Fix $n\ge1$. Since
\[
\inf_{p\in\cP_n}\int_K|p'|^2\,dA=\inf_{q\in\cP_{n-1}}n^2\int_K|q|^2\,dA,
\]
it suffices to prove that
\begin{equation}\label{E:needtoprove}
\lambda:=\inf_{q\in\cP_{n-1}}\int_K|q|^2\,dA\le \frac{\pi}{n}c(K)^{2n}.
\end{equation}
If $K$ has zero area, then this is immediate. Henceforth
we assume that $K$ has positive area. Then,
by standard Hilbert-space arguments, there exists a
unique $q_0\in\cP_{n-1}$ such that $\lambda=\int_K|q_0|^2\,dA$,
and $q_0$ satisfies the orthogonality relations
\begin{equation}\label{E:orthog}
\int_K z^j\overline{q_0(z)}\,dA(z)=0 \quad(0\le j<n-1).
\end{equation}
Using \eqref{E:orthog}, we see that $\lambda$ can also be written as
\begin{equation}\label{E:lambda}
\lambda=\int_Kz^{n-1}\overline{q_0(z)}\,dA(z).
\end{equation}

Given $r>c(K)$, the relations \eqref{E:Chebyshev}
imply that there is a lemniscate $L:=\{z:|P(z)|\le r^d\}$,
where $P$ is a monic polynomial of degree $d$,
such that $K$ is contained in the interior of $L$.
Define $f_0:=1_K\overline{q}_{0}$.
Choose a non-negative radial function $\eta\in C^\infty(\CC)$
of integral $1$,
supported in a disk with centre $0$ whose radius is less than
the distance from $K$ to $\partial L$. Define $f:\CC\to\CC$ by
\[
f(z):=(f_0*\eta)(z)=\int_\CC f_0(z-\zeta)\eta(\zeta)\,dA(\zeta).
\]
Then $f\in C^\infty(\CC)$ and it has compact support
contained in $L$. 
Also, since $f$
is an average of translates of $f_0$, all of which have the same
$L^2(\CC)$-norm as $f_0$,
it follows that $\|f\|_{L^2(\CC)}\le\|f_0\|_{L^2(\CC)}$, in other words
\[
\int_\CC|f|^2\,dA\le \int_\CC|f_0|^2\,dA=\int_{K}|q_0|^2\,dA=\lambda.
\]
Further, as $\eta$ is radially symmetric, 
we have
$\int_\CC(z+\zeta)^j\eta(\zeta)\,dA(\zeta)=z^j$ for all $j\ge0$ and all $z\in\CC$, whence
\begin{equation}\label{E:moments}
\int_\CC z^j f(z)\,dA(z)=\int_K z^j\overline{q_0(z)}\,dA(z) \quad(j\ge0).
\end{equation}

We now define $F:\CC\to\CC$ by
\[
F(z):=-\frac{1}{\pi}\int_\CC\frac{f(w)}{w-z}\,dA(w).
\]
Then $F\in C^\infty(\CC)$, and 
$F_{\overline{z}}=f$ by the Cauchy--Pompeiu formula.
Since $f$ is supported in $L$, the function $F$ is 
holomorphic in $\Omega:=\CC\setminus L$. 
Also, expanding $1/(w-z)$ as a geometric series, we see that,
for $|z|$ large enough,
\[
F(z)=\frac{1}{\pi}\sum_{j\ge0}\Bigl(\int_L w^{j}f(w)\,dA(w)\Bigr)z^{-j-1},
\]
so, from \eqref{E:orthog}, \eqref{E:lambda} and \eqref{E:moments},
we deduce that
\[
F(z)=\frac{\lambda}{\pi} z^{-n} + O(|z|^{-n-1}) \quad (|z|\to\infty).
\]
Lemma~\ref{L:Greenineq} therefore applies, and we deduce that
\begin{equation}\label{E:ineq1}
\int_\Omega |F'(z)|^2\,dA(z)\ge \frac{n\lambda^2}{\pi r^{2n}}.
\end{equation}
Also, from Lemma~\ref{L:energy} and the identity $F_{\overline{z}}=f$, we get
\begin{equation}\label{E:ineq2}
\int_{\Omega}|F'|^2\,dA
\le \int_\CC|F_z|^2\,dA
=\int_\CC|F_{\overline{z}}|^2\,dA
=\int_\CC|f|^2\,dA
\le \lambda.
\end{equation}
Combining \eqref{E:ineq1} and \eqref{E:ineq2}, we obtain
the inequality $\lambda\ge n\lambda^2/\pi r^{2n}$, whence
\[
\lambda\le \pi r^{2n}/n.
\]
Finally, as this holds for each $r>c(K)$, we conclude that
\[
\lambda\le \pi c(K)^{2n}/n,
\]
as was to be proved.
\end{proof}

%%%%%%%%%%%%%%%%%%%%%%%%%%%%%%%%%%

\section{Proofs of Theorem~\ref{T:asymp} and Corollary~\ref{C:asymp}}

The proof of  Theorem~\ref{T:asymp} proceeds via
two lemmas.

\begin{lemma}\label{L:cAbound}
Let $K$ be a compact subset of $\CC$. For  $m\ge0$, set
\begin{equation}\label{E:lambda(K)def}
\lambda_m(K):=\inf_{q\in\cP_m}\int_K|q|^2\,dA.
\end{equation}
Then
\begin{equation}\label{E:cAbound}
\lambda_m(K)\le\frac{\pi}{m+1}c_A(K)^{2m+2}\quad(m\ge0).
\end{equation}
\end{lemma}

\begin{proof}
Fix $m\ge0$.
From Theorem~\ref{T:bound}, we have
\begin{equation}\label{E:cbound}
\lambda_m(K)\le\frac{\pi}{m+1}c(K)^{2m+2}.
\end{equation}
Our goal is to replace $c(K)$ by $c_A(K)$ in the right-hand side.
This is obvious if $A(K)=0$, since then $\lambda_m(K)=0$.
Henceforth we assume that $A(K)>0$.

As $A(K)>0$,  both $\|\cdot\|_{L^2(K)}$ and $\|\cdot\|_{L^\infty(K)}$
are genuine norms on the space of polynomials of degree at most~$m$.
By equivalence of norms on finite-dimensional spaces,
there exists a constant $M=M(m,K)$ such that
\[
\|q\|_{L^\infty(K)}\le M\|q\|_{L^2(K)}
\quad(\deg q\le m).
\]
Hence, if $F$ is a compact subset of $K$ and $\deg q\le m$, then
\begin{align*}
\int_F|q|^2\,dA 
&=\int_K|q|^2\,dA-\int_{K\setminus F}|q|^2\,dA\\
&\ge\int_K|q|^2\,dA-\|q\|_{L^\infty(K)}^2A(K\setminus F)\\
&\ge\int_K|q|^2\,dA-M^2\|q\|_{L^2(K)}^2A(K\setminus F)\\
&=(1-M^2A(K\setminus F))\int_K|q|^2\,dA.
\end{align*}
Taking the infimum
over $q\in\cP_m$ therefore gives that, for all compact $F\subset K$,
\begin{equation}\label{E:lambdaineq}
\bigl(1-M^2A(K\setminus F)\bigr)\lambda_m(K)\le\lambda_m(F)\le\lambda_m(K).
\end{equation}

Let $E$ be a Borel subset of $K$ such that $A(K\setminus E)=0$.
By inner regularity of area measure, there exist compact
sets $F_j\subset E$ such that $A(K\setminus F_j)\to0$.
From \eqref{E:lambdaineq}, it follows that $\lambda_m(F_j)\to\lambda_m(K)$.
Now, applying \eqref{E:cbound} to $F_j$ gives
\[
\lambda_m(F_j)\le \frac{\pi}{m+1}c(F_j)^{2m+2}\le \frac{\pi}{m+1}c(E)^{2m+2}.
\]
Letting $j\to\infty$, we deduce that
\[
\lambda_m(K)\le\frac{\pi}{m+1}c(E)^{2m+2}.
\]
Taking the infimum over all Borel $E\subset K$
with $A(K\setminus E)=0$, we obtain \eqref{E:cAbound}.
\end{proof}

\begin{lemma}\label{L:cAasymp}
Let $K$ and $\lambda_m(K)$ be as in Lemma~\ref{L:cAbound}.
Then
\[ 
\liminf_{m\to\infty}\lambda_m(K)^{1/(2m)}\ge c_A(K).
\]
\end{lemma}

\begin{proof}
Suppose, to obtain a contradiction, that
$\liminf_{m\to\infty}\lambda_m(K)^{1/(2m)}< c_A(K)$.
Choose $\alpha,\beta$ such that
$\liminf_{m\to\infty}\lambda_m(K)^{1/(2m)}<\alpha<\beta< c_A(K)$.
Then there exist a strictly increasing sequence
of positive integers $(m_j)_{j\ge1}$ and polynomials $q_j\in\cP_{m_j}$
such that, for all $j\ge1$,
\begin{equation}\label{E:alpha}
\int_K|q_j|^2\,dA\le \alpha^{2m_j}.
\end{equation}
For each $j\ge1$, define
\[
F_j:=\{z\in K:|q_j(z)|\le\beta^{m_j}\}.
\]
Then
\[
\int_K|q_j|^2\,dA
\ge\int_{K\setminus F_j}|q_j|^2\,dA
\ge\int_{K\setminus F_j}\beta^{2m_j}\,dA
=\beta^{2m_j}A(K\setminus F_j).
\]
Combining with \eqref{E:alpha}, we deduce that
\[
A(K\setminus F_j)\le (\alpha/\beta)^{2m_j}.
\]
Now define
\[ 
E_k:=\bigcap_{j\ge k}F_j 
\quad\text{and}\quad
E:=\bigcup_{k\ge1}E_k.
\]
The sets $E_k$ are compact and increase with $k$. 
Moreover
\[
A(K\setminus E_k)\le \sum_{j\ge k}A(K\setminus F_j)
\le \sum_{j\ge k}(\alpha/\beta)^{2m_j}\to 0 \quad(k\to\infty).
\]
Thus $E$ is a Borel subset  of $K$ and $A(K\setminus E)=0$,
whence
\[
c_A(K)\le c(E).
\]
On the other hand, using \eqref{E:Chebyshev}, we have,
for each $j\ge1$,
\[
c(F_j)\le \|q_j\|_{L^\infty(F_j)}^{1/m_j}\le \beta,
\]
so, using the continuity from below of $c(\cdot)$ for bounded Borel sets (see e.g.\ \cite[Theorem~5.1.3(b)]{Ra95}),
we get
\[
c(E)=\lim_{k\to\infty} c(E_k)\le \liminf_{k\to\infty}c(F_k)\le\beta.
\]
Combining these observations, we deduce that $c_A(K)\le \beta$,
which contradicts our choice of $\beta$.
\end{proof}

\begin{proof}[Proof of Theorem~\ref{T:asymp}]
For each $n\ge1$,
\[
b_n(K)
=\frac{1}{\sqrt{n\pi}}\inf_{p\in \cP_n}\|p'\|_{L^2(K)}
=\sqrt{\frac{n}{\pi}}\inf_{q\in\cP_{n-1}}\|q\|_{L^2(K)}
=\sqrt{\frac{n}{\pi}}\lambda_{n-1}(K)^{1/2},
\]
where $\lambda_{n-1}(K)$ is given by \eqref{E:lambda(K)def}.
Hence, by Lemma~\ref{L:cAbound},
\[
b_n(K)^{1/n}\le c_A(K) \quad(n\ge1),
\]
and by Lemma~\ref{L:cAasymp},
\[
\liminf_{n\to\infty}b_n(K)^{1/n}
=\liminf_{n\to\infty}\Bigl(\sqrt{\frac{n}{\pi}}\lambda_{n-1}(K)^{1/2}\Bigr)^{1/n}
=\liminf_{m\to\infty}\lambda_{m}(K)^{1/(2m)}
\ge c_A(K).
\]
Combining these observations, we deduce that
\[
c_A(K)\le \liminf_{n\to\infty}b_n(K)^{1/n}\le \sup_{n\ge1}b_n(K)^{1/n}\le c_A(K),
\]
from which \eqref{E:asymp} follows.
\end{proof}

It remains to prove Corollary~\ref{C:asymp}. 
As remarked in the introduction, this is an immediate consequence
of Theorem~\ref{T:asymp} once the following lemma has been 
established.

\begin{lemma}\label{L:interior}
Let $K$ be a compact subset of $\CC$. Then 
\begin{equation}\label{E:interior}
c(K^\circ)\le c_A(K)\le c(K).
\end{equation}
\end{lemma}

\begin{proof}
The right-hand inequality in \eqref{E:interior} is obvious, by taking
$E=K$ in \eqref{E:cAdef}. So we concentrate on the left-hand inequality. This inequality is clear if $K^\circ=\emptyset$, so, in what follows, we may as well assume that $K^\circ\ne\emptyset$.

Let $F$ be a non-empty, compact subset of $K^\circ$.
Choose $r>0$ strictly smaller than the distance between $F$ 
and $\partial K$. Thus the closed disk $D(z,r)$ is contained in $K^\circ$ for every $z\in F$. By the mean-value inequality,
every polynomial $q$ satisfies
\[
|q(z)|^2\le \frac{1}{\pi r^2}\int_{D(z,r)}|q|^2\,dA\le \frac{1}{\pi r^2}\int_K|q|^2\,dA
\quad(z\in F).
\]
If $m\ge 1$ and $q\in\cP_m$, then by \eqref{E:Chebyshev}
we have $c(F)\le \|q\|_{L^\infty(F)}^{1/m}$, and hence
\[
c(F)\le \Bigl(\frac{1}{\pi r^2}\int_K|q|^2\,dA\Bigr)^{1/(2m)}.
\]
Taking the infimum over $q\in\cP_m$, 
and then applying Lemma~\ref{L:cAbound}, 
we deduce that
\[
c(F)\le  \Bigl(\frac{1}{(m+1)r^2}c_A(K)^{2m+2}\Bigr)^{1/(2m)} 
\quad(m\ge1).
\]
Letting $m\to\infty$, we obtain
\[
c(F)\le c_A(K).
\]
Finally, taking the supremum over all non-empty compact subsets $F$ of $K^\circ$, we get 
\[
c(K^\circ)\le c_A(K)\qedhere
\]
\end{proof}

%%%%%%%%%%%%%%%%%%%%%%%%%%%%%%%%%%

\section*{AI tools declaration}
In preparing this work, the author used ChatGPT to develop and explore strategies for the proofs of Theorems~\ref{T:bound}
and \ref{T:asymp} and Corollary~\ref{C:asymp}. 
The final mathematical  text was written entirely by the author, who thoroughly corrected, 
verified, and streamlined the proposed strategy. 
The author accepts full responsibility for the accuracy and integrity of the final results presented in this paper.

\bibliographystyle{plain}
\bibliography{bibfile}

\end{document}